\pdfoutput=1

\documentclass[letterpaper, 11pt]{amsart}
\usepackage{amsmath, amsfonts, amssymb, amsthm, amsrefs, array, stmaryrd, graphicx, hyperref, mathrsfs, eucal, caption, soul}

\usepackage[hhmmss]{datetime}

\usepackage{color}
\newcommand{\kibitz}[2]{\ifnum\Comments=1\textcolor{#1}{#2}\fi}

\theoremstyle{plain}
\newtheorem{thm}{Theorem}[section]
\newtheorem{lemma}{Lemma}[section]
\newtheorem{prop}[lemma]{Proposition}

\theoremstyle{definition}

\theoremstyle{remark}
\newtheorem{remark}[lemma]{Remark}
\numberwithin{equation}{section}

\newcommand{\p}{\partial}

\newcommand{\bn}{\begin{enumerate}}
\newcommand{\en}{\end{enumerate}}
\newcommand{\bi}{\begin{itemize}}
\newcommand{\ei}{\end{itemize}}
\newcommand{\bqq}{\begin{eqnarray*}}
\newcommand{\eqq}{\end{eqnarray*}}
\newcommand{\balg}{\begin{align*}}
\newcommand{\ealg}{\end{align*}}

\DeclareMathOperator{\Rm}{Rm}

\begin{document}
\title[On the precise asymptotics of Type-IIb MCF solutions]{On the precise asymptotics of Type-IIb solutions to mean curvature flow}

\author{James Isenberg}
\address{Department of Mathematics, University of Oregon, Eugene, OR 97403, USA}
\email{isenberg@uoregon.edu}

\author{Haotian Wu}
\address{School of Mathematics and Statistics, The University of Sydney, NSW 2006, Australia}
\email{haotian.wu@sydney.edu.au}

\author{Zhou Zhang}
\address{School of Mathematics and Statistics, The University of Sydney, NSW 2006, Australia}
\email{zhou.zhang@sydney.edu.au}

\date{\usdate\today} 

\keywords{Mean curvature flow; noncompact hypersurfaces; Type-IIb curvature blow-up; precise asymptotics.}

\subjclass[2010]{53C44 (primary), 35K59 (secondary)}


\begin{abstract}
In this paper, we study the precise asymptotics of noncompact Type-IIb solutions to the mean curvature flow. Precisely, for each real number $\gamma>0$, we construct mean curvature flow solutions, in the rotationally symmetric class, with the following precise asymptotics as $t\nearrow\infty$: (1) The highest curvature concentrates at the tip of the hypersurface (an umbilical point) and blows up at the Type-IIb rate $(2t+1)^{(\gamma-1)/2}$. (2) In a neighbourhood of the tip, the Type-IIb blow-up of the solution converges to a translating soliton known as the bowl soliton. (3) Near spatial infinity, the hypersurface has a precise growth rate depending on $\gamma$.
\end{abstract}


\maketitle

\section{Introduction}\label{sec:intro}
Given an embedded (more generally immersed) $n$-dimensional hypersurface $\varphi_0: M^n\to\mathbb{R}^{n+1}$ in Euclidean space, consider the one-parameter family of hypersurfaces
$\varphi(t):M^n\to\mathbb{R}^{n+1}$, $t_0<t<t_1$, generated by the mean curvature flow (MCF), which is specified by the evolution equation
\begin{align}\label{eq:mcf}
\partial_t\varphi (p,t) = \vec H,\quad p\in M^n,\quad t_0\leq t<t_1.
\end{align}
Geometrically, MCF deforms a hypersurface in the direction of its mean curvature vector $\vec H$, starting from the initial hypersurface $\varphi\left(t_0\right) = \varphi_0$.

Suppose we have a smooth solution to MCF on a maximal time interval $[0,T)$. Let $h(t)$ denote the second fundamental form of $\varphi(t)(M^n)$. If $T<\infty$, then we say the finite-time solution is
\begin{itemize}
\item \emph{Type-I}\, if $\sup\limits_{M^n\times[0,T)} (T-t)|h|^2<\infty$,

\item \emph{Type-IIa}\, if $\sup\limits_{M^n\times[0,T)} (T-t)|h|^2 = \infty$.
\end{itemize}
If $T=\infty$, then we say the infinite-time solution is
\begin{itemize}
\item \emph{Type-III}\, if $\sup\limits_{M^n\times[0,\infty)} t|h|^2<\infty$,

\item \emph{Type-IIb}\, if $\sup\limits_{M^n\times[0,\infty)} t|h|^2 = \infty$.
\end{itemize}
Analogous categorizations hold for singularities occurring in Ricci flow with $|h|^2$ replaced by $|\Rm|$, where $\Rm$ is the Riemann curvature tensor of a metric evolving by Ricci flow \cite{CLN06}.

The following questions (e.g., \cite{CLN06}*{Problem 8.6} in the context of Ricci flow) are natural: What can be said about the specific blow-up rates of Type-IIa or Type-IIb solutions to MCF? And what can be determined about the asymptotic behaviour of MCF solutions of these types near the maximal time of existence?

In dimension one ($n=1$), MCF of closed embedded curves in $\mathbb{R}^2$ never form a Type-IIa singularity \cite{Gray87}. However, Angenent and Vel\'azquez have constructed a MCF of a closed immersed curve in the plane that forms a Type-IIa singularity with $\sup\limits_{S^1}|h(\cdot,t)|$ blowing up at the rate $\sqrt{\frac{\ln \ln(1/(T-t))}{T-t}}$, which is faster than $(T-t)^{-1/2}$ but slower than any higher power $(T-t)^{-1/2-\epsilon}$ \cite{AV95}. In dimension two or higher ($n\geq 2$), Type-IIa singularities can form in MCF of embedded hypersurfaces. Compact examples with blow-up rates $(T-t)^{-1+1/m}$ for any integer $m\geq 3$ have been constructed by Angenent and Vel\'azquez \cite{AV97}. Noncompact examples with blow-up rates $(T-t)^{-(\gamma+1/2)}$ for any real number $\gamma\geq 1/2$ have been obtained by the authors \cite{IW19, IWZ19a}. There are corresponding results for Type-IIa solutions \cite{AIK15, Wu14} and Type-IIb solutions \cite{Wu19} in Ricci flow.

Type-IIb solutions to MCF are necessarily noncompact because MCF of any compact hypersurface must have finite time of existence by the avoidance principle \cite{CMP15}. In this paper, we construct a class of Type-IIb MCF solutions and describe their precise asymptotic properties. In particular, we determine their Type-IIb curvature blow-up rates, and also determine the behaviour of the solutions near where the curvature is the highest as well as near spatial infinity. Our construction is carried out in the class of complete noncompact hypersurfaces that are smooth, rotationally symmetric, convex\footnote{Throughout this paper, ``convex'' means ``strictly convex''.}, entire graphs with prescribed growth rate at spatial infinity.

Let us set up the notation. For any point $(x_0,x_1,\ldots,x_n)\in\mathbb{R}^{n+1}$, we write
\begin{align*}
x = x_0, \quad r = \sqrt{x_1^2 + \cdots + x_n^2}.
\end{align*}
A noncompact hypersurface $\Gamma$ is said to be rotationally symmetric if
\begin{align*}
\Gamma = \left\{(x_0,x_1,\ldots,x_n) : r = u(x_0), a\leq x_0 < \infty \right\}.
\end{align*}
We assume that $u$ is strictly concave so that the hypersurface $\Gamma$ is (strictly) convex and that $u$ is strictly increasing with $u(a)=0$ and has the asymptotic growth condition $\lim\limits_{x_0\nearrow \infty}u(x_0) = \infty$, which is a necessary condition for the solution to be Type-IIb or Type-III. Indeed, if we assume the contrasting condition that $\lim\limits_{x_0\nearrow \infty}u(x_0) = R$ for some finite positive value of $R$, then the hypersurface is a complete graph over a ball $B^n_{R}$ of radius $R$, and is asymptotic to the cylinder $S^{n-1}_R\times\mathbb{R}$, whose axis is the $x_0$-axis. MCF starting from such a hypersurface escapes to spatial infinity in \emph{finite} time by the work of Saez and Schn\"{u}rer \cite{SS14}. Among MCF solutions of this sort, the authors have exhibited a class of Type-IIa solutions and have described their precise asymptotics in \cite{IW19} and \cite{IWZ19a}.

Returning to our construction of mean curvature flows which exhibit Type-IIb behaviour, we note that the function $u$ is assumed to be smooth except at $x=a$. This particular non-smoothness of $u$ is a consequence of the choice of the coordinates; in fact, as seen below, if the time-dependent flow function $u(x,t)$ is inverted in a particular way, this irregularity is removed. We label the moving point where $u(x,t)=0$ as the \emph{tip} of the hypersurface.

We now denote by $\Gamma_t$ the solution to MCF which starts at a specified choice of the initial embedding of $\Gamma$ (as described above). If we represent $\Gamma_t$ by a graph $r=u(x,t)$, then under MCF the function $u$ satisfies the PDE
\begin{align}\label{eq:u(x,t)}
u_t & = \frac{u_{xx}}{1+u^2_x} - \frac{n-1}{u}.
\end{align}
To help carry out analysis, especially in a neighbourhood of the tip, it is useful to  define the following rescaled quantities 
\begin{align}
\tau & = \log\sqrt{2t+1}, \label{eq:rescale-tau}\\ 
y & = x (2t+1)^{-(\gamma+1)/2}, \label{eq:rescale-y}\\
\phi(y,\tau) & = u(x,t)(2t+1)^{-1/2}, \label{eq:rescale-phi}
\end{align}
where $\gamma$ is a free parameter to be specified. 

To motivate the rescaled quantities \eqref{eq:rescale-tau}--\eqref{eq:rescale-phi}, we first recall that the Type-III blow-up rate is $t^{-1/2}\sim (2t+1)^{-1/2}$, whereas the Type-IIb blow-up rate is faster than the Type-III rate. If we seek MCF solutions with the Type-IIb curvature blow-up rate $(2t+1)^{(\gamma-1)/2}$, where $\gamma>0$, then the part of the hypersurface where the curvature is Type-IIb covers a distance on the order of $(2t+1)^{(\gamma+1)/2}$, as it is moving at a speed proportional to the curvature blow-up rate on the order of $(2t+1)^{(\gamma-1)/2}$. Therefore, we rescale the $x$-coordinate according to \eqref{eq:rescale-y} to bring the hypersurface that is moving to spatial infinity to a finite distance away from the origin. The rescaling \eqref{eq:rescale-phi} is at the Type-III rate and serves as an intermediate step to capture the geometry of the hypersurface where the curvature blow-up is not Type-IIb (cf. Section \ref{formal}). To further study the geometry of the hypersurface near where the curvature is Type-IIb, we introduce a further rescaled quantity (cf.\eqref{eq:coord-z})
\begin{align}\label{eq:rescale-z}
z:=\phi e^{\gamma\tau} = u(2t+1)^{{\gamma-1}/2}.
\end{align}

Substituting the rescaled quantities \eqref{eq:rescale-tau}--\eqref{eq:rescale-phi} into equation \eqref{eq:u(x,t)}, we obtain the following PDE for $\phi(y,\tau)$:
\begin{align}\label{eq:phi(y,tau)}
\left. \p_\tau \right\vert_y \phi & = \frac{e^{-2\gamma\tau} \phi_{yy}}{1+e^{-2\gamma\tau}\phi^2_y} +(\gamma+1) y \phi_y
- \frac{(n-1)}{\phi} - \phi,
\end{align}
where $\left. \p_\tau \right\vert_y $ means taking the partial derivative in $\tau$ while keeping $y$ fixed. We note the resemblance between equation \eqref{eq:phi(y,tau)} and equation (1.3) in \cite{IW19} or \cite{IWZ19a}.

It is useful to invert the coordinates and work with 
\begin{align*}
y( \phi,\tau) & = y\left( \phi(y,\tau), \tau \right);
\end{align*}
this inversion can be done because the hypersurface under consideration is a convex entire graph (opening in the positive $x$-axis). In terms of $y(\phi,\tau)$, the equation corresponding to mean curvature flow, which is equivalent to equation \eqref{eq:phi(y,tau)} and hence equivalent to equation \eqref{eq:u(x,t)}, is the following:
\begin{align}\label{eq:y(phi,tau)}
\left. \p_\tau \right\vert_{\phi} y & = \frac{y_{\phi\phi}}{1 + e^{2\gamma\tau} y^2_\phi} + \left(\frac{n-1}{\phi} + \phi \right) y_\phi - (\gamma+1) y.
\end{align}
Equation \eqref{eq:y(phi,tau)} closely resembles its counterparts in \cite{IW19} and \cite{IWZ19a}. Note that the difference occurs in the zeroth order term and the first order term involving $\phi y_\phi$. As a result, the approach in \cite{IW19} and \cite{IWZ19a} is promising in the current setting.

We use the notation ``$A \sim B$'' to indicate that there exist positive constants $c$ and $C$ such that $cB \leq A \leq CB$. Our main result is the following.

\begin{thm}\label{thmmain}
For any choice of an integer $n\geq 2$ and a pair of real numbers $\gamma > 0$, and $\tilde A>0$, there is a family $\mathscr{G}$ of $n$-dimensional, smooth, rotationally symmetric, strictly convex, entire graphs over $\mathbb{R}^n$ such that MCF evolution $\Gamma_t$ starting at each hypersurface $\Gamma\in\mathscr{G}$ escapes to spatial infinity at $T = \infty$, and  has the following precise asymptotic properties as $t\nearrow\infty$:
\begin{enumerate}
\item The highest curvature occurs at the tip, where $u(x,t)=0$, of the hypersurface $\Gamma_t$, and it blows up at the Type-IIb rate
\begin{align}
\sup\limits_{p\in \mathbb{R}^n}|h(p,t)| & \sim (2t+1)^{(\gamma-1)/2} \quad \text{as } t\nearrow \infty.
\end{align}
\item Near the tip, the Type-IIb blow-up (rescaling) of $\Gamma_t$ converges to a translating bowl soliton; precisely,
\begin{align}
y(e^{-\gamma\tau} z,\tau) & = y(0,\tau) + e^{-2\gamma\tau}\left(\frac{\tilde P\left( (\gamma+1)\tilde A z \right)}{(\gamma+1)\tilde A} + o(1)\right)  \quad \text{as } \tau\nearrow \infty
\end{align}
uniformly on compact $z$ intervals, where $z=\phi e^{\gamma\tau}$, and $\tilde P$ is defined in equation \eqref{eq:tildeP}.
\item Away from the tip and near spatial infinity, the Type-III blow-up of $\,\Gamma_t$  grows at the rate 
\begin{align}\label{eq:yasympt}
y & \sim (\phi^2 + n - 1)^{(\gamma+1)/2} \quad \text{as } \phi \nearrow \infty.
\end{align}
\end{enumerate}
In particular, the solution constructed has the asymptotics predicted by the formal solution described in Section \ref{formal}.
\end{thm}

The asymptotic condition \eqref{eq:yasympt} of our MCF solutions says that $y\sim \phi^{\gamma+1}$ as $\phi\nearrow\infty$. Using the relations $y = x (2t+1)^{-(\gamma+1)/2}$ and $\phi=u(2t+1)^{-1/2}$ to convert $y$ and $\phi$ back to the unscaled coordinates $x$ and $u$ respectively, then 
\begin{align}\label{eq:ugrowth}
\quad x &\sim u^{\gamma+1} \quad \text{as } u\nearrow\infty.
\end{align}
Since $\gamma>0$, the entire graphical hypersurfaces moving by MCF described in Theorem \ref{thmmain} have super-linear growth in $u$ (cf. \cite{Cheng}*{Table 1}). On the other hand, suppose a smooth entire graph, not one of those constructed in Theorem \ref{thmmain}, grows linearly in $u$ near spatial infinity, i.e.,  $x\sim u$ as $u\nearrow\infty$, which is the case if $\gamma=0$ in \eqref{eq:ugrowth}. Then MCF staring from this hypersurface is Type-III with the curvature blow-up rate $(2t+1)^{-1/2}$ by the work of Ecker and Huisken \cite{EH89}. For asymptotic behaviour of Type-III solutions to MCF, we refer the reader to the classical results of Ecker and Huisken \cite{EH89, EH91} and the recent work of Cheng and Sesum \cite{CS18}

The proof of Theorem \ref{thmmain} uses matched asymptotic analysis and barrier arguments for nonlinear PDE, the same strategy that has been implemented for Type-IIa solutions in \cite{IW19} and \cite{IWZ19a}. In Section \ref{formal}, we  describe the construction of the approximate (formal) solutions using formal matched asymptotics. In Section \ref{super-sub}, we use these approximate solutions to construct regional supersolutions and subsolutions to the rescaled PDE. The regional supersolutions and subsolutions are ordered and we patch them together to form barriers for the rescaled PDE in Section \ref{barriers}; a comparison principle for the barriers is also proved there. In Section \ref{existence}, we use these results to complete the proof of Theorem \ref{thmmain}.

\subsection*{Acknowledgements}
J. Isenberg is partially supported by NSF grant PHY-1707427; H. Wu thanks the support by ARC grant DE180101348; Z. Zhang thanks the support by ARC grant FT150100341.

\section{Formal solutions}\label{formal}

\subsection{The formal solutions in the form  $y(z,\tau)$ or $y(\phi, \tau)$}\label{yformal}

To derive the formal solutions, we assume that for $\tau$ large, the terms $\left. \p_\tau \right\vert_{\phi} y$ and $\displaystyle{\frac{y_{\phi\phi}}{1+e^{2\gamma\tau}y^2_\phi}}$ in equation \eqref{eq:y(phi,tau)} are negligible, so the PDE \eqref{eq:y(phi,tau)} is approximated by the ODE
\begin{align}
\label{truncd}
\left(\frac{n-1}{\phi}+\phi\right) \tilde y_\phi - (\gamma+1)\tilde y = 0, 
\end{align}
whose general solution is
\begin{align}\label{eq:truncdsoln}
\tilde y (\phi) & = C_1 \left(\phi^2 + n-1\right)^{(\gamma+1)/2},
\end{align}
where $C_1$ is an arbitrary constant, and $\phi\in [0,\infty)$. For $\gamma>0$, as we presume in this paper, $\tilde y$ is convex and grows super-linearly.

To check the consistency of the assumptions we have made in obtaining the ODE \eqref{truncd}, we substitute the solution $\tilde y$ given in \eqref{eq:truncdsoln} into the quantity $\displaystyle{\frac{y_{\phi\phi}}{1+e^{2\gamma\tau}y^2_\phi}}$, obtaining 
\begin{align*}
\frac{\tilde y_{\phi\phi}}{1 + e^{2\gamma\tau} {\tilde y_\phi}^2} & = \frac{C_1(\gamma+1)(\phi^2+n-1)^{(\gamma-1)/2}(n-1+\gamma\phi^2)}{n-1 + \left[\phi^2 + \boxed{\phi^2e^{2\gamma\tau}}\gamma(1+\gamma)^2(\phi^2 + n -1)^2 C_1^2\right]}.
\end{align*}
This suggests that  $\tilde y$ is a reasonable approximate solution, provided that the boxed term $\phi^2 e^{2\gamma\tau}$ is sufficiently large. 

As in the statement of Theorem \ref{thmmain}, we define
\begin{align}\label{eq:coord-z}
z:=\phi e^{\gamma\tau}.
\end{align}
We label the dynamic (i.e., time-dependent) region where $z=O(1)$ as the \emph{interior region} and call its complement the \emph{exterior region}. Note that the condition $z=O(1)$ is equivalent to the condition $\phi = O\left( e^{-\gamma\tau}\right)$, which corresponds to a region near the tip (at which $\phi = 0$). Since
\begin{align*}
\p_\tau|_z y = \p_\tau|_{\phi} y - \gamma z y_z,
\end{align*}
we obtain from equation \eqref{eq:y(phi,tau)} the evolution equation for $y(z,\tau)$:
\begin{align}\label{eq:y(z,tau)}
\left. \p_\tau\right\vert_z y & = \frac{y_{zz}}{e^{-2\gamma\tau} + e^{2\gamma\tau}y^2_z} +  \left( \frac{n-1}{z}e^{2\gamma\tau} + (1-\gamma) z \right) y_z - (\gamma+1)y.
\end{align}

We consider the ansatz
\begin{align}\label{eq:tipansatz-y}
y = \tilde A + e^{-2\gamma\tau} \tilde F(z,\tau),
\end{align}
where $\tilde A$ is a positive constant and $\tilde F$ is a function to be determined; in particular, $\tilde F(0,\tau)=0$. Before we determine $\tilde F$, let us provide some heuristics for this ansatz. Suppose that we want the Type-IIb curvature blow-up rate $(2t+1)^{(\gamma-1)/2}$, where $\gamma>0$, at the tip of the hypersurface moving by MCF; then we expect the geometry near the tip to be modelled by a translating soliton. So as $t\nearrow\infty$, if we rescale the (x,u)-coordinates by the Type-IIb blow-up rate near the tip, then we expect the hypersurface near the tip to be generated by the profile
\begin{align}\label{eq:heur1}
x(2t+1)^{(\gamma-1)/2} = \tilde A\cdot(2t+1)^\alpha + B\left(u(2t+1)^{(\gamma-1)/2}, t\right),
\end{align}
where $A$ is some constant, $B(\cdot,t)|_{u=0}=0$ for all $t$, and $\alpha$ is a parameter that we determine now. The profile in \eqref{eq:heur1} is equivalent to
\begin{align}\label{eq:heur2}
x = \tilde A\cdot(2t+1)^{\alpha-\frac{\gamma-1}{2}} + (2t+1)^{-(\gamma-1)/2}B\left(u(2t+1)^{(\gamma-1)/2}, t\right).
\end{align}
The tip, where $u=0$ and the curvature is prescribed to be blowing up at the Type-IIb rate $(2t+1)^{(\gamma-1)/2}$, is moving at a speed on the same order and hence covers a distance on the order of $(2t+1)^{(\gamma+1)/2}$. Therefore, comparing $(2t+1)^{(\gamma+1)/2}$ with the coefficient of $\tilde A$ in \eqref{eq:heur2}, we have $\alpha=\gamma$. As a result, \eqref{eq:heur2} becomes
\begin{align*}
x = \tilde A\cdot(2t+1)^{\frac{\gamma+1}{2}} + (2t+1)^{-(\gamma-1)/2}B\left(u(2t+1)^{(\gamma-1)/2}, t\right),
\end{align*}
which, rewritten in the $(y,\tau)$-coordinates, is just the ansatz in \eqref{eq:tipansatz-y}.

Substituting \eqref{eq:tipansatz-y} into equation \eqref{eq:y(z,tau)} yields
\begin{align}\label{eq:tildeAF}
\frac{\tilde F_{zz}}{1 + \tilde F^2_z} + (n-1) \frac{\tilde F_z}{z} & = (\gamma+1) \tilde A + e^{-2\gamma\tau}\left[ (\gamma-1)(z \tilde F_z - \tilde F) + \left.\p_\tau\right\vert_z \tilde F \right]. 
\end{align}
Continuing the formal argument, we assume that for $\tau$ very large, the term in \eqref{eq:tildeAF} with the coefficient $e^{-2\gamma\tau}$ is negligible. Equation \eqref{eq:tildeAF} then reduces  to the ODE
\begin{align}\label{eq:tildeF}
\frac{\tilde F_{zz}}{1+\tilde F^2_z} + (n-1)\frac{\tilde F_z}{z} & =  (\gamma+1) \tilde A.
\end{align}

To solve \eqref{eq:tildeF}  for $\tilde F$, we define $\tilde P(w)$ to be the unique solution to the initial value problem 
\begin{align}\label{eq:tildeP}
\frac{\tilde P_{ww}}{1+(\tilde P_w)^2} + (n-1)\frac{\tilde P_w}{w} = 1,\quad  \tilde P(0) = \tilde P_w(0) = 0.
\end{align}
We then readily verify that if  $\tilde F$ is given by
\begin{align}\label{eq:tildeF(z,tau)}
\tilde F(z,\tau) = \frac{1}{(\gamma+1) \tilde A} \tilde P\left( (\gamma+1) \tilde A z \right) + C(\tau),
\end{align}
where $C(\tau)$ is an arbitrary function of time, then $\tilde F$ satisfies \eqref{eq:tildeF}.

The initial value problem \eqref{eq:tildeP} has been solved in \cite[pp.24--25]{AV97} for general dimensions. It has a unique convex solution defined on $\mathbb{R}$ with the following asymptotics:
\begin{align}\label{eq:asymptotics-tildeP}
\tilde P(w) &= \left\{
\begin{array}{cc}
w^2/2n + o\left(w^2\right), & w\searrow 0; \vspace{6pt} \\
w^2/(2n-2) - \log w + O\left(w^{-2}\right),& w\nearrow\infty.
\end{array}
\right.
\end{align}

Rescaling back to the $(x,u)$-coordinates, then in a neighbourhood of the tip, where $u=0$, we have
\begin{align*}
x & = \tilde A (2t+1)^{(\gamma+1)/2} + C\left(\log\sqrt{2t+1}\right)(2t+1)^{(\gamma-1)/2} + \\
& \quad + \frac{(\gamma+1)\tilde A}{2n} u^2 (2t+1)^{(\gamma-1)/2} + o\left(u^2 (2t+1)^{(\gamma-1)/2} \right).
\end{align*}
For this formal solution, the curvature at the tip is
\begin{align}\label{eq:meancurv}
\left\vert h_{\text{tip}}\right\vert \sim H_{\text{tip}} & \sim \left. \frac{d^2 x}{du^2}\right   \vert_{u=0} = (\gamma+1)\tilde A  (2t+1)^{(\gamma-1)/2}.
\end{align}
It follows then that
\begin{align*}
t |h_{\text{tip}}|^2 & \sim t(2t+1)^{\gamma-1} \sim t^\gamma.
\end{align*}
Therefore, the formal solution is Type-IIb if $\gamma>0$. This argument gives us reason to believe that, once we have constructed the actual solution to MCF presuming that it is controlled by the formal solution, it will have the desired Type-IIb behaviour.

Because the speed of a hypersurface moving by MCF is given by its mean curvature $H$, it follows from \eqref{eq:meancurv} that over the time period $[t_0,\infty)$, the tip of the hypersurface, formally (i.e., as predicted by the formal solution) moves along the $x$-axis to the right from its initial position $x_0$ by the amount $\int_{t_0}^\infty H_{\text{tip}} = +\infty$. So in terms of the $x$-coordinate, the hypersurface evolving by MCF disappears off to spatial infinity as $t\nearrow \infty$. However in terms of  the $y$-coordinate,  provided that $C(\tau)=O(\tau)$, the tip remains a finite distance from the origin for all time $\tau$ since
\begin{align*}
y_0(\tau) & = \tilde A + e^{-2\gamma\tau} C(\tau) \approx \tilde A.
\end{align*}

From this point on, we assume $\gamma$ to be a fixed positive constant.

The formal solutions constructed separately in the interior and the exterior regions each involves a free parameter. Matching the formal solutions on the overlap of the two regions, we can establish an algebraic relationship between these free parameters. 
Setting $z$ equal to a constant $R$, and assuming that $\tau$ is very large, then 
\begin{align}\label{eq:match-formal-int}
y & \approx \tilde A.
\end{align}
In the exterior region, again setting $z=R$ (and therefore $\phi = Re^{-\gamma \tau}$) and again presuming very large $\tau$, we have from \eqref{eq:truncdsoln} that 
\begin{align}\label{eq:match-formal-ext}
y & \approx C_1\left(n - 1\right)^{(\gamma+1)/2}.
\end{align}
Matching \eqref{eq:match-formal-int} with \eqref{eq:match-formal-ext}, we obtain 
\begin{align}\label{eq:matching}
\tilde A = C_1\left(n-1\right)^{(\gamma+1)/2}.
\end{align}

In summary, in the interior region where $z = \phi e^{\gamma\tau} = O(1)$, we blow up the formal solution $u(t,x)$ to MCF at the prescribed Type-IIb rate $(2t+1)^{(\gamma-1)/2}$ and rescale the coordinates in accord with how fast the surface moves under mean curvature flow by setting $y = x (2t+1)^{-(\gamma+1)/2}$. Then in this interior region, the formal solution is given by 
\begin{align*}
y_{form,\;int} & = \tilde A +  e^{-2\gamma\tau} C(\tau) + e^{-2\gamma\tau} \tilde F(z),
\end{align*}
where $\tilde F$ and $ C(\tau)$ (to be specified in Lemma \ref{interior-supersub}) are related to $\tilde P$ as specified in \eqref{eq:tildeF(z,tau)}, and where $\tilde P$ is the solution to the initial value problem \eqref{eq:tildeP}. In the exterior region, where $R e^{-\gamma\tau} \leq \phi < \infty$ for some $R>0$, the formal solution takes the form
\begin{align*}
y_{form,\;ext} & = \frac{\tilde A}{\left(n-1\right)^{(\gamma+1)/2}} \left(\phi^2 + n - 1\right)^{(\gamma+1)/2}.
\end{align*}

We emphasise that the discussion in Section \ref{formal} applies to the formal solutions in the interior region and the exterior region. We show in Section \ref{existence} that the actual MCF solutions we construct also have the asymptotics predicted by the formal solution.


\subsection{The formal solutions revisited in the form  $\lambda(z,\tau)$ or $\lambda(\phi, \tau)$}\label{formallambda}

To prove the main result (Theorem \ref{thmmain}) of this paper, it is useful to also work with the quantity $\lambda := -1/y$, which is a bounded function because of the super-linear growth of the embedded hypersurface corresponding to large values of $y$. The interval of $\phi$ remains noncompact; in fact $\phi\in\mathbb{R}$.

Under MCF, the evolution equation for $\lambda$ is readily obtained by substituting the definition of $\lambda$ into \eqref{eq:y(phi,tau)}:
\begin{align}\label{eq:lambda(phi,tau)}
\left. \p_\tau \right\vert_\phi \lambda & = \frac{\lambda_{\phi\phi}-2\lambda^2_\phi/\lambda}{1+e^{2\gamma\tau}\lambda^2_\phi/\lambda^4} + \left(\frac{n-1}{\phi} + \phi\right)\lambda_\phi + (\gamma+1)\lambda.
\end{align}
The class of MCF solutions we consider here correspond to (even) solutions of equation \eqref{eq:lambda(phi,tau)} subject to the following effective boundary conditions: the super-linear growth of $y$ implies that $\lim_{|\phi|\nearrow\infty}\lambda(\phi, \tau)=0$. By the rotational symmetry, $\lambda_\phi (0,\tau) = 0$.

As in the previous analysis in terms of $y$, it is useful here to use the dilated spatial variable $z=\phi e^{\gamma\tau}$. The evolution equation for $\lambda(z,\tau)$ then takes the form
\begin{align}\label{eq:lambda(z,tau)}
\left. \p_\tau \right\vert_z \lambda & = \frac{e^{2\gamma\tau}(\lambda_{zz}-2\lambda^2_z/\lambda)}{1+e^{4\gamma\tau}\lambda^2_z/\lambda^4} + e^{2\gamma\tau}(n-1)\frac{\lambda_z}{z} + (1-\gamma)z\lambda_z + (\gamma+1)\lambda.
\end{align}

We now construct the formal solutions in terms of $\lambda(z,\tau)$ or $\lambda(\phi,\tau)$, using arguments very similar to those used above in terms of $y$. 

In the interior region, where $z=O(1)$, we use the ansatz
\begin{align*}
\lambda=-A+e^{-2\gamma\tau}F(z),
\end{align*}
where $A$ is a positive constant and $F$ is a function which we now determine. Substituting this ansatz into equation \eqref{eq:lambda(z,tau)}, we find that 
$F$ must satisfy
\begin{align}
\label{Flambda}
e^{-2\gamma\tau}\left( -2\gamma F  + \left.\partial_\tau\right\vert_z F \right) & = \frac{F_{zz} - 2e^{-2\gamma\tau} F^2_z/(-A+e^{-2\gamma\tau}F)}{1+F^2_z/(-A+e^{-2\gamma\tau} F)^4}\\
\nonumber & \quad + (n-1)\frac{F_z}{z} - (\gamma+1)A \\
\nonumber &\quad + e^{-2\gamma\tau}\left[(1-\gamma)z F_z + (\gamma+1)F \right].
\end{align}
Assuming, in our formal argument, that the terms with coefficient $e^{-2\gamma\tau}$ in equation \eqref{Flambda} can be ignored for large $\tau$, then \eqref{Flambda} reduces to the following ODE for $F$:
\begin{align}\label{eq:ODE-F}
\frac{F_{zz}}{1+F^2_z/A^4} + (n-1) \frac{F_z}{z} = (\gamma+1)A.
\end{align}
To find solutions to \eqref{eq:ODE-F}, we rescale $F$ according to
\begin{equation}
\label{FP}
F(z)= \frac{A^3}{\gamma+1} P\left(z(\gamma+1)/A\right),
\end{equation}
and determine that $P(w)$, where $w := z(\gamma+1)/A$, satisfies the ODE
\begin{align*}
\frac{P_{ww}}{1+P_w^2} + (n-1)\frac{P_{w}}{w} = 1.
\end{align*}
Subject to the initial conditions $P(0)=P_w(0)=0$ which come naturally from the geometry of our hypersurface, we can solve for $P$ uniquely (cf. equation \eqref{eq:tildeP}). Moreover, the asymptotic expansions of $P(w)$ are known:
\begin{align*}
P(w) &= \left\{
\begin{array}{cc}
\frac{1}{2n} w^2 + o\left(w^2\right), & w\searrow 0;\vspace{6pt} \\
\frac{1}{2(n-1)} w^2 - \log w + O\left(w^{-2} \right),& w\nearrow\infty.
\end{array}
\right.
\end{align*}
Consequently, the asymptotic expansions of $F(z)$ are as follows:
\begin{align}\label{eq:asymp-F}
F(z) &= \left\{
\begin{array}{cc}
\frac{(\gamma+1)A}{2n} z^2 + o\left(z^2\right), & z\searrow 0;\vspace{6pt} \\
\frac{(\gamma+1)A}{2(n-1)} z^2 - \frac{A^3}{(\gamma+1)}\log\left( (\gamma+1)z/A\right) + O\left( z^{-2}\right),& z\nearrow\infty.
\end{array}
\right.
\end{align}

In the exterior region, examining  the evolution of $\lambda(\phi,\tau)$ as governed by the PDE \eqref{eq:lambda(phi,tau)}, we assume, as part of the formal argument, that the term $\displaystyle{\frac{\lambda_{\phi\phi}-2\lambda^2_\phi/\lambda}{1+e^{2\gamma\tau}\lambda^2_\phi/\lambda^4}}$ is negligible for $\tau$ large. Then any solution of the ODE
\begin{align}\label{eq:ODE-lambdabar}
\left(\frac{n-1}{\phi} + \phi \right)\bar\lambda_\phi + (\gamma+1)\bar\lambda = 0
\end{align}
is an approximate solution to equation \eqref{eq:lambda(z,tau)}. We can solve for $\bar\lambda(\phi)$ explicitly,
\begin{align}\label{eq:lambdabar}
\bar\lambda(\phi) & = C \left(\phi^2 +n-1\right)^{-(\gamma+1)/2}
\end{align}
for an arbitrary constant $C$.



\section{Supersolutions and subsolutions}\label{super-sub}

For a differential equation of the form $\mathcal{D}[\psi]=0$, a function $\psi^+$ is a \emph{supersolution} if $\mathcal{D}[\psi^+]\ge 0$, while $\psi^-$ is a \emph{subsolution} if $\mathcal{D}[\psi^-]\le 0$. If there exist a supersolution $\psi^+$ and a subsolution $\psi^-$ for the differential operator $\mathcal{D}$, and if they satisfy the inequality $\psi^+ \ge \psi^-$, then they are called \emph{upper and lower barriers}, respectively. If $\mathcal{D}[\psi]=0$ admits solutions, then the existence of barriers $\psi^+\geq \psi^-$ implies that there exists a solution $\psi$ with $\psi^+ \ge \psi \ge \psi^-$. 

In this section, we construct subsolutions and supersolutions for the rescaled MCF PDE in the interior and the exterior regions separately, and then in the next section combine them to get the global barriers along the flow.

\subsection{Interior region}

In the interior region, we work with  $\lambda(z,\tau)$, and with the corresponding  MCF equation \eqref{eq:lambda(z,tau)}. Hence, we work with the quasilinear parabolic operator
\begin{align}\label{eq:Tz}
\mathcal{T}_z[\lambda] : = & \left. \p_\tau \right\vert_z \lambda - \frac{e^{2\gamma\tau}(\lambda_{zz}-2\lambda^2_z/\lambda)}{1+e^{4\gamma\tau}\lambda^2_z/\lambda^4} - e^{2\gamma\tau}(n-1)\frac{\lambda_z}{z} + (\gamma-1)z\lambda_z - (\gamma+1)\lambda,
\end{align}
for which we seek a subsolution and a supersolution. The result is the following.

\begin{lemma}\label{interior-supersub}
For an integer $n\geq 2$, a real number $\gamma>0$, and any pair of positive real numbers $A^{\pm}$, we define the functions $F^{\pm}$ to be the solutions to equation \eqref{eq:ODE-F} with the constants $A=A^{\pm}$ respectively.

For any fixed constants $R_1>0$, $B^{\pm}$ and $E^{\pm}$, there exist functions $Q^{\pm}:\mathbb{R}\to\mathbb{R}$, constants $D^{\pm}$, and a sufficiently large $\tau_1<\infty$ such that the functions
\begin{align}\label{eq:int-supersub}
\lambda^{\pm}_{int} (z,\tau) & := -A^{\pm} + e^{-2\gamma\tau}F^{\pm}(z) + e^{-2\gamma\tau} \left(B^{\pm}\tau + E^{\pm}\right)  + \tau e^{-4\gamma\tau}D^{\pm}Q^{\pm}(z)
\end{align}
are a supersolution ($+$) and a subsolution ($-$), respectively, of $\mathcal{T}_z[\lambda]=0$ on the interval $0 \leq |z| \leq R_1$ for all $\tau\geq\tau_1$.

The functions $Q^{\pm}$ depend on $A^{\pm}$ and $F^{\pm}(z)$ respectively. The constants $D^{\pm}$ depend on $n$, $\gamma$, $A^{\pm}$ and $B^{\pm}$ respectively, and on $R_1$.
\end{lemma}

\begin{proof}
Careful inspection of the argument in Lemma 3.1 of \cite{IW19} (or \cite{IWZ19a}) reveals that the definition of the function $Q$ depends on the second order term and the first order term with $\lambda_z/z$ in the operator $\mathcal{T}_z$ defined in \eqref{eq:Tz}. The terms with $z\lambda_z$ and $\lambda$ in our operator, although different here, only change the constant $C$ in the inequalities for $D^+$ and $D^-$, but do not change the construction argument in the proof of Lemma 3.1 of \cite{IW19} (or \cite{IWZ19a}). So the lemma is proved.
\end{proof}

\begin{remark}
Any positive constants $A^+$ and $A^-$ work in Lemma \ref{interior-supersub}. In Lemma \ref{int-order}, we choose ordered constants $A^+$ and $A^-$ so that $\lambda^+_{int}$ and $\lambda^-_{int}$ are ordered.
\end{remark}

\subsection{Exterior region}

In the exterior region, we work with the quantity $\lambda(\phi,\tau)$, and with the corresponding MCF equation \eqref{eq:lambda(phi,tau)}. Hence, defining the quasilinear parabolic operator
\begin{align}\label{eq:Fphi}
\mathcal{F}_\phi[\lambda] &: =  \left. \p_\tau \right\vert_\phi \lambda - \frac{\lambda_{\phi\phi}-2\lambda^2_\phi/\lambda}{1+e^{2\gamma\tau}\lambda^2_\phi/\lambda^4} - \left(\frac{n-1}{\phi}+\phi \right)\lambda_\phi - (\gamma+1)\lambda,
\end{align}
we seek a subsolution and a supersolution for this operator. The existence of these is proven in the following lemma.

\begin{lemma}\label{exterior-supersub}
For an integer $n\geq 2$ and a real number $\gamma>0$, we define\footnote{This definition is consistent with \eqref{eq:lambdabar}; therefore $\bar\lambda$ satisfies equation \eqref{eq:ODE-lambdabar}.}
\begin{align}\label{eq:barlambda}
\bar\lambda = \bar\lambda(\phi) := (\phi^2 +n-1)^{-(\gamma+1)/2}.
\end{align}

For any positive constants $c^\pm$, there exists an even function $\psi:\mathbb{R} \to \mathbb{R}$ such that for any fixed $R_2>0$, there exist a pair of constants $b^\pm$ and sufficiently large $\tau_2<\infty$, for which
\begin{align}\label{eq:ext-supersub}
\lambda^{\pm}_{ext} & = \lambda^{\pm}(\phi,\tau) := -c^{\pm} \bar\lambda(\phi) + b^{\pm} e^{-2\gamma\tau}\psi(\phi)
\end{align}
are a supersolution ($+$) and a subsolution ($-$), respectively, of $\mathcal{F}_\phi[\lambda]=0$ over the region $R_2 e^{-\gamma\tau} \leq \vert\phi\vert < \infty$ for all $\tau\geq\tau_2$. The constants $b^\pm$ depend on $n, \gamma$, $R_2$, and $c^\pm$, respectively.
\end{lemma}

\begin{proof} The functions involved are all even in $\phi$, so we need only consider $\phi\geq 0$.
	
Going through the proof of Lemma 3.2 of \cite{IW19} (or \cite{IWZ19a}) shows that the definition of the function $\psi$ does not depend on the second order term in the operator $\mathcal{F}_\phi$. Here, we define $\psi$ to be any solution of the ODE
\begin{align}\label{eq:ODE-psi}
-(1+3\gamma) \psi - \left( \frac{n-1}{\phi} + \phi \right)\psi' & = \Lambda,
\end{align}
where
\begin{align*}
\Lambda & :=-\frac{(-\bar\lambda'') - 2(-\bar\lambda')^2/(-\bar\lambda)}{(-\bar\lambda')^2/(-\bar\lambda)^4}\\
& = -\frac{\gamma\phi^2 + n - 1}{(\gamma+1)\phi^2}\bar\lambda^3\\
& < 0
\end{align*}
for all $\phi\in\mathbb{R}$. The general solution $\psi$ to this ODE is
\begin{align}\label{eq:soln-psi}
\psi(\phi) & = \bar\lambda^3\left\{ \frac{1-\gamma}{2(1+\gamma)}+C_1(\phi^2+n-1) + \frac{(\phi^2+n-1)}{2(n-1)(\gamma+1)}\left(\log(\phi^2) - \log(\phi^2+n-1)\right)\right\},
\end{align}
where $C_1$ is an arbitrary constant.

Applying the operator $\mathcal{F}_\phi$ defined in \eqref{eq:Fphi} to the function $\lambda^{+}_{ext}$ from \eqref{eq:ext-supersub}, we obtain (omitting the superscript ``+'' and the subscript ``ext'' to simplify the notation)
\begin{align*}
e^{2\lambda\tau}\mathcal{F}_\phi\left[\lambda^+_{ext}\right] & = II + b\left[ -(1+3\gamma) \psi - \left( \frac{n-1}{\phi} + \phi \right)\psi' \right],
\end{align*}
where $\psi$ solves the ODE \eqref{eq:ODE-psi} and
\begin{align*}
II & = - \frac{\lambda_{\phi\phi}-2\lambda^2_\phi/\lambda}{e^{-2\gamma\tau}+\lambda^2_\phi/\lambda^4}\\
& = c^3\Lambda \left[1 + O\left( e^{-2\gamma\tau}b\psi/(c\bar\lambda), e^{-2\gamma\tau}b\psi'/(c\bar\lambda'), e^{-2\gamma\tau}b\psi''/(c\bar\lambda'') \right) \right].
\end{align*}
So then
\begin{align*}
e^{2\lambda\tau}\mathcal{F}_\phi\left[\lambda^+_{ext}\right] & = \Lambda \left\{c^3\left[1 + O\left( e^{-2\gamma\tau}b\psi/(c\bar\lambda), e^{-2\gamma\tau}b\psi'/(c\bar\lambda'), e^{-2\gamma\tau}b\psi''/(c\bar\lambda'') \right) \right] + b\right\}.
\end{align*}

Let $c_k$ and $d_k$, where $k=0,1,2$, denote constants that depend on the now-fixed constants $n$ and $\gamma$. From \eqref{eq:barlambda} and \eqref{eq:ext-supersub} we have as $\phi\nearrow \infty$ the following asymptotics
\begin{align*}
\psi/\bar \lambda & = (\phi^2+n-1)^{-\gamma}\left[ c_0 + o( 1 ) \right],\\
\psi'/\bar \lambda' & = (\phi^2+n-1)^{-\gamma}\left[ c_1 + o( 1 ) \right],\\
\psi''/\bar \lambda'' & = (\phi^2+n-1)^{-\gamma}\left[ c_2 + o( 1 ) \right],
\end{align*}
and as $\phi\searrow 0$, the following asymptotics
\begin{align*}
\psi/\bar \lambda & = d_0 \log(\phi^2) + O(1),\\
\psi'/\bar \lambda' & = \phi^{-2} \left[ d_1 + O\left(\phi^2\log(\phi^2)\right) \right],\\
\psi''/\bar \lambda'' & = \phi^{-2} \left[ d_2 + O\left(\phi^2\log(\phi^2)\right) \right].
\end{align*}

The above asymptotics imply the following estimates.
If $\delta\leq\phi<\infty$ for some fixed $\delta>0$ (e.g., $\delta=1/2$), then we have
\begin{align*}
\left\vert O\left( e^{-2\gamma\tau}b\psi/(c\bar\lambda), e^{-2\gamma\tau}b\psi'/(c\bar\lambda'), e^{-2\gamma\tau}b\psi''/(c\bar\lambda'')\right)  \right\vert \leq b M_1 e^{-2\gamma\tau}
\end{align*}
for some constant $M_1$. Consequently, we choose $\tau_2$ sufficiently large so that $M_1e^{-2\gamma\tau}<1/(2c^3)$ for all $\tau\geq\tau_2$ (recall that $c$ is fixed), so then
\begin{align*}
e^{2\lambda\tau}\mathcal{F}_\phi\left[\lambda^+_{ext}\right] & \geq \Lambda \left\{b + c^3\left(1+b M_1 e^{-2\gamma\tau} \right)\right\}\\
& >0
\end{align*}
for $\delta\leq\phi<\infty$ if $b$ satisfies $b<-2c^3/3$. 

If $R_2e^{-2\gamma\tau}\leq \phi \leq \delta$, then we have
\begin{align*}
\left\vert O\left( e^{-2\gamma\tau}b\psi/(c\bar\lambda), e^{-2\gamma\tau}b\psi'/(c\bar\lambda'), e^{-2\gamma\tau}b\psi''/(c\bar\lambda'')\right)  \right\vert \leq b M_2 R_2^{-2}
\end{align*}
for some constant $M_2$, and so
\begin{align*}
e^{2\lambda\tau}\mathcal{F}_\phi\left[\lambda^+_{ext}\right] & \geq \Lambda \left\{b + c^3\left(1+b M_2 R_2^{-2} \right)\right\}\\
& >0
\end{align*}
for any $b$ satisfying $b<-c^3/(1+c^3M_2R_2^{-2})$. 

Therefore, there exists
\begin{align*}
b^{+}\leq \min\left\{-\frac{2\left(c^+\right)^3}{3},\,  \frac{-\left(c^+\right)^3}{1+\left(c^+\right)^3M_2R_2^{-2}} \right\}
\end{align*}
such that $\lambda^+_{ext}$ defined in \eqref{eq:ext-supersub} is a supersolution of $\mathcal{F}_\phi[\lambda]=0$ on the interval $R_2 e^{-2\gamma\tau}\leq\phi<\infty$ for all $\tau\geq\tau_2$.

By a similar argument, there exists
\begin{align*}
b^{-}\geq \max\left\{-2(c^-)^3,\,  \frac{-\left(c^-\right)^3}{1-\left(c^-\right)^3M_2R_2^{-2}} \right\}
\end{align*}
such that $\lambda^-_{ext}$ defined in \eqref{eq:ext-supersub} is a subsolution of $\mathcal{F}_\phi[\lambda]=0$ on the interval $R_2 e^{-2\gamma\tau}\leq\phi<\infty$ for all $\tau\geq\tau_2$.

Therefore, the lemma is proved.
\end{proof}

\begin{remark}
From the proof of Lemma \ref{exterior-supersub}, we always have $b^+<0$, and we can pick $b^- > 0$.
\end{remark}



\section{Upper and lower barriers}\label{barriers} 

According to Lemmata \ref{interior-supersub} and \ref{exterior-supersub}, if we choose $R_2<R_1$, then there is an overlap of the interior and exterior regions where both $\lambda^{\pm}_{int}$ and $\lambda^{\pm}_{ext}$ are defined. In order to show that the regional supersolutions $\lambda^+_{ext}$ and $\lambda^+_{int}$ together with the regional subsolutions $\lambda^-_{ext}$ and $\lambda^-_{int}$ collectively provide upper and lower barriers according to the standard $\sup$ and $\inf$ constructions for the rescaled PDE of MCF, we need to show the following:
\begin{itemize}

\item[(i)] in each region, $\lambda ^-_{int}\leqslant \lambda ^+_{int}$ and $\lambda ^-_{ext}\leqslant \lambda ^+_{ext}$;

\item[(ii)] $\lambda ^+_{int}$ and $\lambda ^+_{ext}$ patch together; i.e., $\sup\{\lambda ^+_{int}, \lambda ^+_{ext}\}$ takes the values of $\lambda ^+_{int}$ and then $\lambda ^+_{ext}$ when moving from the interior to the exterior region. Similarly for $\lambda ^-_{int}$ and $\lambda ^-_{ext}$;

\item[(iii)] the patched supersolutions and subsolutions have the required comparison relation throughout, i.e., $\lambda ^-_{ext}\leqslant \lambda ^+_{int}$ and $\lambda ^-_{int}\leqslant \lambda ^+_{ext}$ wherever they are defined, in addition to the inequalities included in (i). 

\end{itemize}

Item (i) follows from the following two lemmata, which are proved by the same line of logic used to prove Lemmata 4.1 and 4.2 in \cite{IW19} or \cite{IWZ19a}. 

\begin{lemma}
\label{int-order}
For $A^{-}>A^{+}$, there exists $\tau_3\geq\tau_1$ such that
\begin{align*}
\lambda^{\pm}_{int} = -A^{\pm} + e^{-\tau}F^{\pm}(z) + \left(B^{\pm}\tau+E^{\pm}\right) e^{-\tau} +  \tau e^{-2\tau}D^{\pm}Q^{\pm}(z)
\end{align*}
satisfy $\lambda^{-}_{int} < \lambda^{+}_{int}$ for $|z|\leqslant R_1$ and for $\tau\geq\tau_3$.
\end{lemma}

\begin{lemma}
\label{ext-order}
For $c^{+}>c^{-}$, there exists $\tau_4\geq\tau_1$ such that
\begin{align*}
\lambda^{\pm}_{ext} & = \bar\lambda^\pm(\phi) + b^{\pm} e^{-\tau}\psi(\phi)
\end{align*}
(as in Lemma \ref{exterior-supersub}) satisfy $\lambda^{-}_{ext}< \lambda^{+}_{ext}$ for $R_2e^{-\tau/2}\leqslant |\phi| < \infty$ and $\tau\geq\tau_4$.
\end{lemma}

To justify (ii), i.e., the patching of supersolutions (or subsolutions) by taking infimum (or supremum), we recall that Lemma \ref{interior-supersub} holds for any $R_1>0$ and Lemma \ref{exterior-supersub} holds for any $R_2>0$. Below, we  choose $1\ll R_2<R_1$ and patch together $\lambda ^+_{int}$ and $\lambda ^+_{ext}$, and $\lambda ^-_{int}$ and $\lambda ^-_{ext}$ in the region defined by $\{ R_2 < z < R_1 \}$. To this end, we need the following lemma.

\begin{lemma}\label{patch}
For a fixed integer $n\geq 2$ and a fixed real number $\gamma >0$, set
\begin{align}\label{eq:tildebeta}
\tilde\beta := \frac{(n-1)^{-3(\gamma+1)/2}}{\gamma+1} > 0.
\end{align} 
Let $\lambda^{+}_{int}$ and $\lambda^{-}_{int}$ be as discussed in Lemmata \ref{interior-supersub} and \ref{int-order}, and $\lambda^{+}_{ext}$ and $\lambda^{-}_{ext}$ as discussed in Lemmata \ref{exterior-supersub} and \ref{ext-order}. There are properly chosen constants $A^{\pm}>0, B^{\pm}, b^+<0, b^->0$ and $c^{\pm}>0$ satisfying
\begin{align}
A^{\pm}& = c^{\pm}(n-1)^{-(\gamma+1)/2}>0, \label{eq:choiceA}\\
B^{\pm} & = - \gamma b^{\pm}\tilde\beta, \label{eq:choiceB}
\end{align}
such that for some sufficiently large $R_1$ and $R_2$ for Lemmata \ref{interior-supersub} and \ref{exterior-supersub}, we have for $\tau\geq \tau_5$ with some sufficiently large $\tau_5$, the functions
\begin{align*}
\lambda^{+}_{int} - \lambda^{+}_{ext}, \quad\lambda^{-}_{ext} - \lambda^{-}_{int}
\end{align*} 
both strictly increase from negative to positive in the $z$-interval $(R_2, R_1)$.
\end{lemma}

\begin{proof}
We prove the Lemma for $\phi \in [0, \infty)$; the proof for $\phi \in (-\infty, 0]$ follows as a consequence of the evenness of the function $\phi$.

In the interior region, using the asymptotic expansion of $F(z)$ in \eqref{eq:asymp-F}, we have that as $z\nearrow \infty$,
\begin{align*}
\lambda^{+}_{int} & = - A^{+} + B^{+}\tau e^{-2\gamma\tau} +\\
&\quad e^{-2\gamma\tau}\left[\frac{(\gamma+1)A^{+}}{2(n-1)}z^2 - \frac{(A^+)3}{(\gamma+1)}\log\left((\gamma+1)z/A^+\right)+ E^{+} + O\left(z^{-2}\right)\right]\\
&\quad + D^{+} \tau e^{-4\gamma\tau}Q^{+}(z).
\end{align*}
In the exterior region, using the asymptotic expansion that readily follows from the explicit expression for $\psi(\phi)$ in \eqref{eq:soln-psi}, we have that as $\phi\searrow 0$, 
\begin{align*}
\lambda^{+}_{{ext}} & = -c^+\bar\lambda\left(ze^{-\gamma\tau}\right) + b^+ e^{-2\gamma\tau}\psi\left(ze^{-\gamma\tau}\right)\\
& = -\frac{c^+}{(n-1)^{(\gamma+1)/2}} + \frac{c^+(\gamma+1)}{2(n-1)^{(\gamma+3)/2}}z^2 e^{-2\gamma\tau} + O\left( z^4e^{-4\gamma\tau}\right)\\
&\quad -\gamma b^+ \tau e^{-2\gamma\tau}\left[\tilde\beta + O\left(z^2e^{-2\gamma\tau}\right)\right]\\
& \quad + b^+ e^{-2\gamma\tau} \left[\tilde\beta\log|z| + d + O\left( z^2 e^{-2\gamma\tau}(1+\log|z|+\tau)\right)\right],
\end{align*}
where $\tilde{\beta}$ is defined in \eqref{eq:tildebeta} and $d = (n-1)^{-3(\gamma+1)/2}\left(\frac{1-\gamma}{2+2\gamma}+C_1(n-1)\right)$ is an arbitrary constant. 
It then follows that
\begin{align*}
\lambda^{+}_{int} - \lambda^{+}_{ext} & = \left( - A^{+} + c^{+}(n-1)^{-(\gamma+1)/2} \right) + \left(B^+ + \gamma b^+\tilde\beta\right)\tau e^{-2\gamma\tau}\\
& \quad + e^{-2\gamma\tau}\left[\frac{(\gamma+1)A^+}{2(n-1)}z^2 - \frac{c^+(\gamma+1)}{2(n-1)^{(\gamma+3)/2}}z^2\right]\\
& \quad + e^{-2\gamma\tau}\left[ -\frac{(A^+)^3}{(\gamma+1)}\log|z| - b^+\tilde\beta \log|z| + E^{+} -b^+d + O(z^{-2}) \right]\\
& \quad + O\left(e^{-4\gamma\tau}\left(\tau Q^+(z)+z^2(1+\log|z|+\tau)\right)\right).
\end{align*}
By \eqref{eq:choiceA} and $\eqref{eq:choiceB}$, the first two lines in the above expression are zero and we get
\begin{align*}
e^{2\gamma\tau}(\lambda^{+}_{int} - \lambda^{+}_{ext}) & = \left(-\frac{(A^+)^3}{(\gamma+1)} - b^+\tilde\beta \right) \log|z| +O\left(z^{-2}\right) + E^{+} -b^+d \\
& \quad + O\left(e^{-2\gamma\tau}\left(\tau Q^+(z)+z^2(1+\log|z|+\tau)\right)\right), 
\end{align*}
with its derivative with respect to $z$ given by
\begin{align*}
e^{2\gamma\tau}(\lambda^{+}_{int} - \lambda^{+}_{ext})_z & = \left(-\frac{(A^+)^3}{(\gamma+1)} - b^+\tilde\beta \right) z^{-1} +O\left(z^{-3}\right)\\
& \quad + O\left(e^{-2\gamma\tau}\left(\tau (Q^+)'+z(1+\log|z|)\right)\right). 
\end{align*}
So far, we have chosen that $A^+>0$ and $\tilde\beta>0$. Moreover, we can choose $b^+$ such that $\left(-\frac{(A^+)^3}{(\gamma+1)} - b^+\tilde\beta \right)>0$. Additionally, we know that Lemma \ref{exterior-supersub} holds for $\lambda^+_{{ext}}$ with the current choice of $b^+$. Consequently, we have the following observations regarding $\lambda^{+}_{int}-\lambda^{+}_{ext}$ for sufficiently large $\tau$:
\begin{enumerate}
\item The function $e^{2\gamma\tau}(\lambda^{+}_{int} - \lambda^{+}_{ext})$ is smooth and strictly increasing with respect to $z$ on some interval $(R, 10R)$ where $R\gg 1$. 
	
\item Fixing the constant $d$, then by adjusting the value of $E^{+}$, which is a constant independent of $\tau$, we can make sure $e^{2\gamma\tau}(\lambda^{+}_{int} - \lambda^{+}_{ext})$ has only one zero at some $z\in(R, 10R)$ while (1) holds.	
\end{enumerate}
Letting $R_2 = R$ and $R_1 = 10R$, we have that $\lambda^+_{{int}}-\lambda^+_{{ext}}$ strictly increases from negative to positive in the $z$-interval $(R_2, R_1)$

In the same way, we can deal with $\lambda^-_{int}$ and $\lambda^-_{ext}$. Clearly, we can choose the same interval $(R_2, R_1)$ by adjusting the previously chosen one if necessary. So the lemma is proved.
\end{proof}

\begin{remark}
The choices of $A^{\pm}$ and of $c^{\pm}$ in Lemma \ref{patch} are compatible with those of Lemmata \ref{int-order} and \ref{ext-order}.
\end{remark}

We can now patch the regional supersolutions and subsolutions, thereby producing the global supersolution and the global subsolution, which are consequently upper and lower barriers. More precisely, for $|\phi|\in[0,\infty)$ and $\tau\geq\tau_5$, we define $\lambda^{+}:=\lambda^{+}(\phi,\tau)$ by
\begin{align}\label{eq:upperbarrier}
\lambda^{+} : = \left\{
\begin{array}{cc}
\lambda^{+}_{int}, & |\phi|\leq R_2 e^{-\gamma\tau}, \vspace{6pt} \\ 
\inf\left\{\lambda^{+}_{int}, \lambda^{+}_{ext} \right\}, & R_2 e^{-\gamma\tau}\leq |\phi| \leq R_1 e^{-\gamma\tau}, \vspace{6pt}\\
\lambda^{+}_{ext} , & R_1 e^{-\gamma\tau} \leq |\phi| < \infty,
\end{array}
\right.
\end{align}
and similarly we define $\lambda^{-}:=\lambda^{-}(\phi,\tau)$ by
\begin{align}\label{eq:lowerbarrier}
\lambda^{-} : = \left\{
\begin{array}{cc}
\lambda^{-}_{int}, & |\phi|\leq R_2 e^{-\gamma\tau}, \vspace{6pt}\\
\sup\left\{\lambda^{-}_{int}, \lambda^{-}_{ext} \right\}, & R_2 e^{-\gamma\tau}\leq |\phi| \leq R_1 e^{-\gamma\tau}, \vspace{6pt}\\
\lambda^{-}_{ext} , & R_1e^{-\gamma\tau} \leq |\phi| < \infty,
\end{array}
\right.
\end{align}
where the above Lemma \ref{patch} is crucial in justifying the legitimate transition from the interior construction to the exterior construction. The properties of the barriers $\lambda^{\pm}$ are a straightforward consequence of the above construction and are summarized in the following proposition.
\begin{prop}\label{prop-patch}
For a fixed integer $n\geq 2$, let $\lambda^{+}$ and $\lambda^{-}$ be defined as in \eqref{eq:upperbarrier} and \eqref{eq:lowerbarrier} respectively. There exists a sufficiently large $\tau_0$ such that the following hold true for $\phi\in\mathbb{R}$ and $\tau\geq\tau_0$:
\begin{enumerate}
\item[(B1)] $\lambda^{+}$ and $\lambda^{-}$ are a supersolution ($+$) and a subsolutions ($-$) for equation \eqref{eq:lambda(phi,tau)} respectively;
\item[(B2)] $\lambda^{-} < \lambda^{+}$; 
\item[(B3)] near $\phi = 0$, $\lambda^{\pm} = \lambda^{\pm}_{int}$, and near $\phi=\infty$, $\lambda^{\pm} = \lambda^{\pm}_{ext}$; 
\item[(B4)] for any $\tau\in[\tau_0,\infty)$, $\lim\limits_{|\phi|\nearrow \infty} \lambda^{\pm} = 0$.
\end{enumerate}
\end{prop}

We now prove a comparison principle for any pair of smooth functions such that one of them is a subsolution of equation $\mathcal{F}_\phi[\lambda] = 0$ (cf. \eqref{eq:lambda(phi,tau)}) and the other is a supersolution of the same equation. These functions need not be $\lambda^{\pm}$ constructed above, but of course, the proposition is used to justify that $\lambda^{\pm}$ are indeed barriers. We point out that our $\lambda^{\pm}$ are continuous and piecewise smooth on their domains of definition.

\begin{prop}\label{comparison}{(Comparison principle for $\mathcal{F}_{\phi}[\lambda]=0$)}
For a fixed  integer $n\geq 2$, a fixed real number $\gamma>0$, and some $\bar\tau\in(\tau_0, \infty)$, suppose that $\zeta^{+}$, $\zeta^{-}$ are any smooth non-positive supersolution ($+$) and subsolution ($-$) (not necessarily those constructed in Proposition \ref{prop-patch}) of the equation $\mathcal{F}_{\phi}[\lambda]=0$, respectively. Assume that
\begin{itemize}
\item[(C1)] $\zeta^{-}(\phi,\tau_0) < \zeta^{+}(\phi,\tau_0)$ for $\phi\in\mathbb{R}$,\\

\item[(C2)] $\lim\limits_{|\phi|\nearrow \infty}\left(\zeta^{-}(\phi,\tau) - \zeta^{+}(\phi,\tau)\right) \leq 0$ for $\tau\in[\tau_0,\bar\tau]$,\\

\end{itemize}

Then $\zeta^{-}(\phi,\tau_0) \leq \zeta^{+}(\phi,\tau_0)$ for $(\phi,\tau)\in \mathbb{R}\times [\tau_0,\bar\tau]$.
\end{prop}

\begin{proof}
Assumptions (C1)--(C2) imply that given any choice of $\epsilon>0$, there exists $R=R(\epsilon)$ such that
\begin{align}\label{eq:zetabc}
\zeta^{-}(\phi,\tau) - \zeta^{+}(\phi,\tau) \leq  \epsilon, \quad\text{for }|\phi|=R \text{ and any }\tau\in[\tau_0,\bar\tau].
\end{align}
Let us define $v:= e^{-\mu \tau}(\zeta^+ - \zeta^{-}) + 2\epsilon$, where $\mu$ is to be determined. Then assumptions (C1)--(C2) and \eqref{eq:zetabc} imply that $v$ satisfies the following conditions
\begin{itemize}
\item[(C1')] $v(\phi,\tau_0)>0$ for $\phi\in[-R,R]$,\\

\item[(C2')] $v(\phi,\tau_0)\geq\epsilon>0$ for $|\phi|=R$ and any $\tau\in[\tau_0,\bar\tau]$.
\end{itemize}

We claim that $v > 0$ on $(\phi,\tau)\in [-R,R]\times [\tau_0,\bar\tau]$, where $\bar\tau<\infty$ and $\tau_0$ can be chosen to be positive because we are interested in the asymptotics near the first singular time $T$ and we can always start the flow at $T-1+\delta$ for some fixed $\delta>0$.

To prove the above claim, we suppose the contrary. Then it follows from assumptions (C1')--(C2') that there must be a first time $\tau_*\in(\tau_0,\bar\tau)$ and an interior point $\phi_*\in[-R,R]$ such that
\begin{align*}
v(\phi_*, \tau_*) = 0.
\end{align*}
Moreover, at $(\phi_*,\tau_*)$, we have
\begin{align*}
\p_\tau\vert_{\phi} v  \leq 0, &  \quad\quad\quad \zeta^{+}_{\phi\phi}  \geq \zeta^{-}_{\phi\phi}, \\
\zeta^{+}_{\phi}  = \zeta^{-}_{\phi}, & \quad\quad\quad \zeta^{+} - \zeta^{-}  =  -2\epsilon e^{\mu \tau_*}.
\end{align*}
Consequently at $(\phi_*,\tau_*)$, we have
\begin{align*}
0 & \geq e^{\mu \tau_*} \p_\tau\vert_{\phi} v \\
& = \p_\tau\vert_{\phi} (\zeta^{+} - \zeta^{-})  - \mu \tau_* (\zeta^{+}-\zeta^{-})\\
& \geq (\zeta^+ - \zeta^-) \left\{ \left. \frac{\zeta^-_{\phi\phi}(\zeta^++\zeta^-)\left((\zeta^+)^2+(\zeta^-)^2\right)}{\left(e^{2\gamma\tau}(\zeta^+_\phi)^2 + (\zeta^+)^4\right)\left(e^{2\gamma\tau}(\zeta^+_\phi)^2 + (\zeta^-)^4\right)}\right\vert_{(\phi_*,\tau_*)} \right.\\
& \quad \left. + \left.\frac{2(\zeta^+_\phi)^2}{\zeta^+\zeta^-(1+e^{2\gamma\tau}(\zeta^{+}_\phi)^2/(\zeta^{+})^4)}\right\vert_{(\phi_*,\tau_*)} + (\gamma+1) - \mu \tau_*  \right\}\\
& = -2\epsilon e^{\mu \tau_*} \left\{ (\text{bounded term independent of $\mu$}) - \mu \tau_* \right\}\\
& \geq -2\epsilon e^{\mu \tau_*} \left\{ (\text{bounded term independent of $\mu$}) - \mu \tau_0 \right\}.
\end{align*}
We recall that $\epsilon>0$ is fixed. If we choose $\mu$ sufficiently large, then  at $(\phi_*,\tau_*)$ we have 
\begin{align*}
0 & \geq \p_\tau\vert_{\phi} v > 0,
\end{align*}
which is a contradiction. Hence, the claim is true. In the proof of the claim, $\mu$ may depend on $\zeta^+$, $\zeta^-$ and $\bar\tau$, but not on $\epsilon>0$. Therefore, letting $\epsilon\to 0$, the proposition follows.
\end{proof}

\begin{remark} Proposition \ref{comparison} applies to continuous piecewise smooth functions as discussed in \cite{IW19, IWZ19a}.
\end{remark}


\section{Proof of the main theorem}\label{existence}
In this section,  we prove the main result of this paper.

\begin{proof}[Proof of Theorem \ref{thmmain}] Let $n\geq 2$ and $\gamma>0$. Let $\tau_0\geq \tau_5$, where $\tau_5$ is given in Lemma \ref{patch}.

We first patch the formal solutions in the interior and the exterior regions at $\tau=\tau_0$ to obtain a continuous piecewise smooth function $\widehat\lambda(\phi)$ defined for all $\phi\in\mathbb{R}$. Given $\tilde A>0$, we let $A :=1/\tilde A$ and set
\begin{align*}
c:=A(n-1)^{(\gamma+1)/2},\quad C_0 := A -c\left(R_1^2e^{-2\gamma\tau_0}+n-1\right)^{-(\gamma+1)/2}.
\end{align*}
Recalling  that $z=\phi e^{-\gamma\tau}$, we define
\begin{align*}
\widehat\lambda_0(\phi) := \left\{
\begin{array}{lr}
-A + e^{-2\gamma\tau_0}F(z)- e^{-2\gamma\tau_0}F(R_1) + C_0, & 0\leq|z|\leq R_1, \vspace{6pt}\\
-c(\phi^2+n-1)^{-(\gamma+1)/2}, & R_1e^{-\gamma\tau_0}\leq|\phi|<\infty,
\end{array}
\right.
\end{align*}
where $R_1$ is defined in Lemma \ref{patch}.

For any $\epsilon>0$ sufficiently small, by taking $\tau_0$ large enough, we can construct barriers $\lambda^\pm(\phi,\tau)$, as discussed in Section 4 with $A^+<A<A^-$, such that for all $\phi\in\mathbb{R}$, 
\begin{align*}
\lambda^-(\phi,\tau_0) < \widehat \lambda_0(\phi) < \lambda^+(\phi,\tau_0),\quad |\lambda^+ - \lambda^-|<\epsilon.
\end{align*}
For each choice of the continuous piecewise smooth function $\hat \lambda_0$, arguing as in \cite{IWZ19a}*{Lemma 5.4}, there is an open\footnote{The locally open condition applies near where we smooth the corner. The prescribed geometries near the tip and the spatial infinity are unaffected.} set, in the $C^0_{loc}$-topology, of smooth functions, all of which are trapped between $\lambda^(\phi,\tau_0)$ and $\lambda^+(\phi,\tau_0)$. Collecting all such trapped smooth functions, we obtain a set $\mathscr{G}$ of functions which are trapped between the subsolution $\lambda^{-}(\phi,\tau_0)$ and the supersolution $\lambda^{+}(\phi,\tau_0)$.

We now record some properties of $\lambda_0\in\mathscr{G}$. Recall that a rotationally symmetric hypersurface, whose profile function is $u(x)$, has principal curvatures  
$$\kappa_1=\cdots=\kappa_{n-1}=\frac{1}{u(1+u^2_x)^{1/2}}, \quad \kappa_n=-\frac{u_{xx}}{(1+u^2_x)^{3/2}},$$
where the first $n-1$ indices correspond to the rotation and $n$ corresponds to the graph direction. Let us define 
\begin{align*}
\mathfrak{R}:=\frac{\kappa_n}{\kappa_1} = -\frac{uu_{xx}}{1+u^2_x}.
\end{align*}
Then we readily verify that $\mathfrak{R}$ is scaling invariant and satisfies the evolution equation
\begin{align}\label{eq:pde-ratioR}
\mathfrak{R}_t=\frac{\mathfrak{R}_{xx}}{1+u^2_x}-\frac{2u_x}{u(1+u^2_x)}(1-\mathfrak{R})\mathfrak{R}_x+\frac{2u^2_x}{u^2(1+u^2_x)}[(1-\mathfrak{R}^2)+(n-2)(1-\mathfrak{R})].
\end{align}
It is straightforward to check that for any $\lambda_0\in\mathscr{G}$, we have $\mathfrak{R}\leq C$ for some positive constant $C$ (fixed in this proof) and $\lim\limits_{|x|\nearrow\infty}\mathfrak{R}=0$.
Also, cf. \cite{IWZ19a}*{Lemma5.4}, we can choose $R_1$ with 
\begin{align}\label{eq:R1}
100 C R_1^{-4} < (\gamma+1/2)\tilde A.
\end{align} 

The hypersurface corresponding to a choice of $\lambda_0\in\mathscr{G}$ is a smooth, entire, strictly convex (rotationally symmetric) hypersurface. Since it is a locally Lipschitz continuous entire graph over $\mathbb{R}^n$, MCF starting from such a hypersurface exists for all time by a classic result of Ecker and Huisken \cite{EH91}*{Theorem 5.1}. Then by the comparison principle (Proposition \ref{comparison}), the MCF solution $\lambda(\phi,\tau)$ with the initial condition $\lambda_0$ is always trapped between the barriers $\lambda^{\pm}(\phi,\tau)$ which, as $t\nearrow\infty$, both have the same asymptotic behaviour at spatial infinity given by
\begin{align*}
	\lambda \sim (\phi^2+n-1)^{-(\gamma+1)/2}.
\end{align*}
Hence, item (3) of Theorem \ref{thmmain} follows. Item (3) implies that the asymptotic growth of the hypersurface flowing by MCF is preserved, and hence $\lim\limits_{|x|\nearrow\infty}\mathfrak{R}(t)=0$ for all $t\geq t_0$. Then applying the maximum principle to equation \eqref{eq:pde-ratioR} (cf. \cite{IWZ19a}*{Lemma 5.1}), we have 
\begin{align}\label{eq:RleqC}
\mathfrak{R}\leq C
\end{align}
for all $t\geq t_0$ under MCF.

Now we proceed to justify the accurate curvature blow-up rate and singularity model as stated in items (1) and (2) of Theorem \ref{thmmain}. To study the behaviour of such a MCF solution near the tip as $\tau\nearrow\infty$, we work with $y(z,\tau)$ instead of $\lambda(z,\tau)$. Recall that $y(z,\tau)$ evolves by equation \eqref{eq:y(z,tau)}. Recall that $\tilde A = -1/A$. Define $\tilde p(z,\tau)$ by the relation
\begin{align}\label{eq:ptilde}
y(\phi,\tau) = \tilde A + e^{-\tau} \tilde p(z,\tau).
\end{align}
Then $\tilde p(z,\tau)$ satisfies the PDE, $\mathcal{B}[\tilde p] = 0$ where
\begin{align*}
\mathcal{B}[\tilde p] & = a - \left(\frac{\tilde{p}_{zz}}{1+\tilde{p}^2_z} + \frac{n-1}{z}\tilde p_{z}\right) + e^{-\tau}\left( \left.\p_\tau\right\vert_z \tilde p + z\tilde p_z - \tilde p \right).
\end{align*}

Recall that $\phi(y, \tau)$ and $y(\phi, \tau)$ denote the functions along the flow which are inverse to each other. Define 
\begin{align*}
y^{(0)}(\phi) & :=c^{-1}\left(\phi^2 + n -1 \right)^{(\gamma+1)/2},\\
\phi^{(0)}(y) & :=\sqrt{(cy)^{2/(\gamma+1)}-(n-1)}.
\end{align*}
Let $\lambda^{(0)}(\phi) :=  -1/ y^{(0)}(\phi)$. By the uniformity in the construction of the initial hypersurface and the barriers in terms of $\lambda_0$ and $\lambda^{\pm}$ , we have as $\tau\to \infty$ 
$$\lambda(\phi, \tau)\to \lambda^{(0)}(\phi)$$
locally uniformly for $\phi\in [0, \infty)$ by the comparison principle for the equation of $\lambda(\phi,\tau)$, as in Lemma 7.1 of \cite{AV97}. In particular, we obtain the uniform closeness to the barriers on the initial hypersurface by direct construction, whereas in \cite{AV97} the estimates use an Exit Lemma (cf. \cite[Lemma 3.1]{AV97}) and the information of the neck region, which is no longer present in our case. Therefore, 
$$y(\phi, \tau)\to y^{(0)}(\phi)$$
locally uniformly for $\phi\in [0, \infty)$. 

We then prove the following result corresponding to Lemma 7.2 in \cite{AV97}.  
\begin{lemma}[Type-IIb blow-up]\label{lem:type2}
Recall the function $\tilde P$ defined in \eqref{eq:tildeP} which forms part of a formal solution to MCF. We have the following asymptotic behaviour of $\tilde p$:
\begin{align}
\lim\limits_{\tau\nearrow\infty} \left( \tilde p(z,\tau) - \tilde p(0,\tau) \right) = \frac{1}{(\gamma+1)\tilde A} \tilde P\left( (\gamma+1)\tilde A z \right)
\end{align}
uniformly on compact $z$ intervals.
\end{lemma}
\begin{proof}[Proof of Lemma \ref{lem:type2}]

We show that $\tilde p(z,\tau)$ converges uniformly to $\tilde{P}\left((\gamma+1)\tilde A z \right)$ as $\tau\to\infty$ for bounded $z\geq 0$. By the Fundamental Theorem of Calculus and $\tilde P(0) = 0$, it suffices to show that $\tilde{p}_z(z,\tau)$ converges uniformly to $\tilde{P}'\left((\gamma+1)\tilde{A} z\right)$ as $\tau\nearrow\infty$ for bounded $z\geq 0$. To this end, let us introduce a new ``time'' variable 
\begin{align*}
s:=\frac{e^{2\gamma\tau}}{2\gamma}.
\end{align*}
In terms of $s$,  the function $\tilde p$ defined in \eqref{eq:ptilde} satisfies the PDE
\begin{align}\label{eq:p_s}
\left. \p_s \right\vert_z \tilde p & = \frac{\tilde{p}_{zz}}{1+\tilde{p}^2_z} + \frac{n-1}{z}\tilde p_{z} - (\gamma+1)\tilde A + \frac{\gamma-1}{2\gamma} \frac{1}{s} (\tilde p - z\tilde p_z).
\end{align}
The quantity $q(z,s):=\tilde{p}_z(z,\tau)$ satisfies $\mathcal{B}[q]=0$, where
\begin{align}\label{eq:P[q]}
\mathcal{B}[q] & = \frac{\p q}{\p s} + \frac{\gamma-1}{2\gamma} \frac{1}{s} z q_z - \frac{\p}{\p z}\left( \frac{q_z}{1+q^2} + \frac{n-1}{z} q \right).
\end{align} 
We note that equations \eqref{eq:p_s} and \eqref{eq:P[q]} are similar to equations (7.13) and (7.14) in \cite{AV97}. Indeed, we see that the coefficient $\frac{\gamma-1}{2\gamma}$ here is replaced by $\frac{m-1}{m-2}$ in \cite{AV97}. 

The proof in \cite[pp.51--58]{AV97} applies to our case \emph{mutatis mutandis} (see a summary of the argument for the convergence of $q$ in \cite{IWZ19a}*{Appendix B}) except for the construction of a supersolution for the equation $\mathcal{B}[q]=0$. The construction in \cite{AV97} (p.53) uses the fact that the constant coefficient is $\frac{m-1}{m-2}>0$ for $m\geq 3$, but in this paper the constant coefficient is $\frac{\gamma-1}{2\gamma}$, where $\gamma>0$, and it can be negative if $\gamma\in(0,1)$. Here, we construct a supersolution on the domain $\Sigma_\eta:=\{(z,s):0\leq z\leq \sqrt{\eta s}\}$, where $\eta>0$ is a constant to be chosen, as follows. Let $Q(z)=\tilde{P}'(z)$ and define
\begin{align*}
q^{+}(z,s) = Q(z)+\frac{1}{s}L^+(z),
\end{align*} 
where $L^+(z)$ solves the ODE (which is just equation (7.17) in \cite{AV97})
\begin{align}
\frac{(L^+)'(z)}{1+Q(z)^2} + \left(\frac{n-1}{z} -\frac{2Q(z)Q'(z)}{(1+(Q(z)^2))^2} \right) = (\theta^+) \frac{z^2}{2}
\end{align}
with $L(0)=0$ and $\theta^+$ is a constant to be chosen. Then following the rest of the argument on p.52--53 in \cite{AV97}, we have
\begin{align*}
\mathcal{B}\left[Q(z)+L^+(z)/s\right]& \geq \left(-\theta^+ + \frac{\gamma-1}{2\gamma}Q'(z) - C\eta \right),
\end{align*}
where $C$ is some constant and there is another constant $C_1<\infty$ such  that 
\begin{align*}
\sup_{z>0} \left\vert \frac{\gamma-1}{2\gamma}Q'(z) \right\vert\leq C_1
\end{align*}
because $\tilde{P}(z)$ solves the ODE-IVP \eqref{eq:tildeP}. Then we can choose $\eta$ small so that $-C\eta\geq -1$ and so
\begin{align*}
\mathcal{B}\left[Q(z)+L^+(z)/s\right]& \geq \left(-\theta^+ - C_1 - 1\right) \ >0,
\end{align*}
on $\Sigma_\eta$ as long as we choose $-\theta^+ > C_1 + 1$.

So the lemma is proved.
\end{proof}

Lemma \ref{lem:type2} implies that a smooth convex MCF solution expressed in $y(z,\tau)$ satisfies the following asymptotics: on a compact $z$ interval (in the interior region), as $\tau\nearrow\infty$, 
\begin{align*}
y (z,\tau) & = \tilde A - e^{-\gamma\tau} \tilde{p}(0,\tau) + e^{-2\gamma\tau} \frac{1}{(\gamma+1)\tilde A} \tilde P\left( (\gamma+1)\tilde A z \right)\\
& = y(0,\tau) + e^{-2\gamma\tau} \frac{1}{(\gamma+1)\tilde A} \tilde P\left( (\gamma+1)\tilde A z \right).
\end{align*}
So item (2) of Theorem \ref{thmmain} is proved. This expansion is indeed valid for higher order derivatives in light of higher order estimates involved in the proof of Lemma \ref{lem:type2}.

Item (2) implies that at $t\nearrow \infty$, for $z\in[0, R_1]$, our MCF solution necessarily blows up at the rate predicted by the formal solution $e^{-2\gamma\tau} \frac{\tilde P\left( (\gamma+1)\tilde A z \right)}{(\gamma+1)\tilde A}$ (cf. Section \ref{formal}), for which $\mathfrak{R}\leqslant 1$ according to \cite{ADS19}*{Lemma 3.5}. In particular, $\kappa_1=\kappa_n=(\gamma+1)\tilde A (2t+1)^{(\gamma-1)/2}$ at the tip. If $z=u(2t+1)^{(\gamma-1)/2}\geqslant R_1$, then
\begin{align*}
\kappa_1=u^{-1}(1+u_x^2)^{-1/2} \leqslant R_1^{-1}(2t+1)^{(\gamma-1)/2}.
\end{align*}
Then it follows from \eqref{eq:R1} and \eqref{eq:RleqC} that $\mathfrak{R}=\kappa_n/\kappa_1\leqslant CR_1^{-3}$, and hence
\begin{align*}
\kappa_n \leqslant CR_1^{-3} \kappa_1 \leqslant  C R_1^{-4}(2t+1)^{(\gamma-1)/2}< (\gamma+1)\tilde A(2t+1)^{(\gamma-1)/2}.
\end{align*}
So as $t\nearrow\infty$, the highest curvature of this MCF solutions occurs at the tip and blows up at the Type-IIb rate $(2t+1)^{-(\gamma-1)/2}$, which proves part (1) of Theorem \ref{thmmain}.

Therefore, Theorem \ref{thmmain} is proved.
\end{proof}

\bibliography{mcf-type2b_bib}
\bibliographystyle{amsplain}

\end{document}